\documentclass[11pt]{amsart} 

\usepackage[dvips]{graphics} 
\usepackage{color,amssymb,amsbsy,amsmath,amsfonts,amssymb,
amscd,epsfig,times,graphics}

\newtheorem{theo}{Theorem}[section] 
\newtheorem{prop}[theo]{Proposition}

\newtheorem{lem}[theo]{Lemma}

\theoremstyle{definition}
\newtheorem{defi}[theo]{Definition}

\numberwithin{equation}{section}

\newcommand{\C}{\mathbb C} 
\newcommand{\R}{\mathbb R} 
 
\newcommand{\B}{\mathbb B}

\newcommand{\partx}{\partial/\partial x}
\begin{document} 
\begin{abstract}
Let $D=\{\rho<0\}$ be a smooth domain of finite type in an almost complex
manifold  $(M,J)$ of real dimension four. We assume that the defining function $\rho$ is $J$-plurisubharmonic on a neighborhood of $\overline{D}$. We study the asymptotic behavior of pseudoholomorphic discs contained in 
the domain $D$. 
\end{abstract} 
\title[A note on pseudoconvex domains of finite type]
{A note on the geometry of pseudoconvex domains of finite type in almost complex manifolds} 
\author{Florian Bertrand}
\address{University of Wisconsin, 480 Lincoln Drive, Madison, Wisconsin 53706, USA}
\email{bertrand@math.wisc.edu}
\maketitle 

\section*{Introduction} 

A well known problem is to determine which smooth  domains into an almost complex manifold $(M,J)$ are locally complete hyperbolic in the sense of Kobayashi. It seems natural to make some curvature assumptions on such domains as any $J$-pseudoconcavity boundary point is at finite Kobayashi distance to the interior of the domain \cite{iv-ro}. 
I.Graham \cite{gra} gave asymptotic estimates of the Kobayashi pseudometric for strictly pseudoconvex bounded domains into $(\C^n,J_{st})$ and proved the complete hyperbolicity of those domains. In case $(M,J)$ is  any almost complex manifold, similar results were provided  by S.Ivashkovich-J.-P.Rosay \cite{iv-ro} and H.Gaussier-A.Sukhov \cite{ga-su}. 

The situation is far from being so clear when the domains are weakly pseudoconvex as the geometry of their boundary is much more complicated; and the question of whether a smooth weakly pseudoconvex domain is locally complete hyperbolic or not is still open, even in $(\C^2,J_{st})$. However, D.Catlin \cite{ca} obtained local estimates similar to those obtained in \cite{gra} on smooth pseudoconvex domains of finite type in $(\C^2,J_{st})$, implying  their local complete hyperbolicity (see also \cite{be2}). 

In the present paper we study this question for smooth $J$-pseudoconvex domains of finite type into a four dimensional almost complex manifold. In \cite{ber}, we described locally finite type domains $D=\{\rho<0\}$ where 
$\rho$ is a smooth defining function for $D$ and $J$-plurisubharmonic (see Proposition $2.6$ in \cite{ber}). 
More precisely, if $D=\{\rho<0\}\subset \C^2$ and if the origin $0 \in \partial D$ is of finite type $2m$, we proved that there is a change of coordinates in a neighborhood of the origin such that the structure $J$ and the 
function $\rho$ can be locally written:
\begin{equation}\label{eqstrintro}
J=\left(\begin{array}{ccccc} 
a_1& b_1 & 0 & 0 \\
c_1 & -a_1 & 0 & 0\\
0  & 0 & a_2 & b_2\\
0 & 0 & c_2 & -a_2\\ 
\end{array}\right)=J_{st}+O\left(|z_2|\right),
\end{equation} 
and,
\begin{equation}\label{eqintro}
\rho=\Re e z_2+H_{2m}\left(z_1,\overline{z_1}\right)+\widetilde{H}(z_1,z_2)+
O\left(|z_1|^{2m+1}+|z_2||z_1|^m+|z_2|^2\right)
\end{equation} 
where  $H_{2m}$ is a homogeneous polynomial of degree $2m$, subharmonic which  is not harmonic and 
$\widetilde{H}(z_1,z_2)=\Re e \displaystyle \sum_{k=1}^{m-1} \rho_{k} z_1^k z_2.$

In this paper, we suppose that the harmonic term $\widetilde{H}(z_1,z_2)$ in the above expression (\ref{eqintro}) 
is identically zero and we prove:  
\vskip 0,3cm

\begin{theo}\label{theo1} 
Let $J$ be a smooth almost complex structure defined on $\R^4$. 
Let $D=\{\rho<0\}$ be a  domain of finite type in $(\R^4,J)$,  
where $\rho$ is a smooth defining function of $D$, $J$-plurisubharmonic 
in a neighborhood of $\overline{D}$. We suppose that $J$ and $\rho$ satisfy respectively (\ref{eqstrintro}) and 
(\ref{eqintro}). Moreover we assume that $\widetilde{H}(z_1,z_2)$ in (\ref{eqintro}) is identically zero. Then there exists a neighborhood $U$ of the origin for which $0$  is at infinite distance from points in $D\cap U$.
\end{theo}

\vskip 0,3cm

The proof of this theorem is inspired by \cite{iv-ro} and is based on the construction of good $J$-plurisubharmonic functions whose use is significant in almost complex manifolds.

\vskip 0,7cm

\begin{flushleft}
{\it Acknowledgments.} I would like to thank J. Bland and J.-P.Rosay for helpful discussions. I am particularly
indebted to J.-P.Rosay who pointed out an erroneous argument in the previous version of my paper.
\end{flushleft}

\section{Preliminaries}
We denote by $\Delta$ the unit disc of $\C$ and by $\Delta_{r}$ 
the disc of $\C$ centered at the origin of radius $r>0$.
\subsection{Almost complex manifolds and pseudoholomorphic discs}

An {\it almost complex structure} $J$ on a real smooth manifold $M$ is a $\left(1,1\right)$ tensor field
 which  satisfies $J^{2}=-Id$. We suppose that $J$ is smooth.
The pair $\left(M,J\right)$ is called an {\it almost complex manifold}. We denote by $J_{st}$
the standard integrable structure on $\C^{n}$ for every $n$.
A differentiable map $f:\left(M',J'\right) \longrightarrow \left(M,J\right)$ beetwen two almost complex manifolds is said to be 
 {\it $\left(J',J\right)$-holomorphic}  if $J\left(f\left(p\right)\right)\circ d_{p}f=d_{p}f\circ J'\left(p\right),$ 
for every $p \in M'$. In case  $M'=\Delta \subset \C$, such a map is called a {\it pseudoholomorphic disc}.  
If $f:\left(M,J\right)\longrightarrow M'$ is a diffeomorphism, we define an almost complex structure, $f_{*}J$,  
on $M'$ as the
 direct image of $J$ by $f$ :
$$f_{*}J\left(q\right):=d_{f^{-1}\left(q\right)}f\circ J\left(f^{-1}\left(q\right)\right)\circ d_{q}f^{-1},$$ 
 for every  $q \in M'$.

\vspace{0,5cm}

The following lemma (see \cite{ga-su}) states that locally any almost
 complex manifold can be seen as the unit ball of 
$\C^n$ endowed with a small smooth perturbation of the standard
 integrable structure $J_{st}$. 
\begin{lem}\label{ilemloc}
Let $\left(M,J\right)$ be an almost complex manifold, with $J$ of class
 $\mathcal{C}^{k}$, $k\geq 0$. 
Then for every point $p \in M$ and every $\lambda_0 > 0$ there exist a neighborhood $U$ of $p$ and
 a coordinate diffeomorphism $z: U \rightarrow \B$ centered a $p$ (ie $z(p) =0$) such that  the
direct image of $J$ satisfies $z_{*}J\left(0\right) = J_{st}$ and
$||z_*\left(J\right) - J_{st}||_{\mathcal{C}^k\left(\bar {\B}\right)}
 \leq \lambda_0$.
\end{lem}
This is simply done by considering a local chart $z: U \rightarrow \B$ centered a $p$ (ie $z(p) =0$), composing it 
with a linear diffeomorphism to insure $z_{*}J\left(0\right) = J_{st}$ and  dilating coordinates.   

\vspace{0,5cm}

So let  $J$ be an almost complex structure defined in  a neighborhood $U$ of the origin in $\R^{2n}$, and such 
that $J$ is sufficiently closed to the standard structure in uniform norm on the closure $\overline{U}$ of $U$. 
The $J$-holomorphy equation for a pseudoholomorphic disc 
$u : \Delta \rightarrow U \subseteq \R^{2n}$ is given by  
\begin{equation}\label{eqholo}
\frac{\partial u}{\partial y}-J\left(u\right)\frac{\partial u}{\partial x}=0.
\end{equation}
According to \cite{ni-wo}, for every $p \in M$, there is a neighborhood $V$ of zero in $T_{p}M$, such that for every $v \in V$, 
there is a $J$-holomorphic disc $u$ satisfying $u\left(0\right)=p$ and $d_{0}u\left(\partx\right)=v$.

\subsection{Levi geometry}

Let $\rho$ be a $\mathcal{C}^2$ real valued function on
a smooth almost complex manifold  $\left(M,J\right).$
We denote by $d^c_J\rho$ the differential 
form defined by 
$$d^c_J\rho\left(v\right):=-d\rho\left(Jv\right),$$
where  $v$ is a section of $TM$. 
The {\it Levi form} of $\rho$ at a point $p\in M$ and a vector 
$v \in T_pM$ is defined by
\begin{equation*}
\mathcal{L}_J\rho\left(p,v\right):=d\left(d^c_J\rho\right)(p)\left(v,J(p)v\right)=dd^c_J\rho(p)\left(v,J(p)v\right).
\end{equation*}

\vspace{0,5cm}

The next proposition is useful in order to compute the Levi form (see \cite{iv-ro}).   
\begin{prop}\label{proplevi}\mbox{ }
Let $p\in M$ and $v\in T_pM$. Then  
$$\mathcal{L}_J\rho\left(p,v\right)=\Delta \left(\rho \circ u\right)
 \left(0\right),$$ 
where $u : \Delta \rightarrow \left(M,J\right)$ is any $J$-holomorphic
 disc satisfying 
$u\left(0\right)=p$ and $d_0u\left(\partial/\partial_x\right)=v$.
\end{prop}

If $\mathcal{L}_{J}\rho(p,v)\geq 0$ for every $p \in M$ and every $v \in T_pM$, we say that $\rho$ is
 {\it $J$-plurisubharmonic}. It is a well know fact that $\rho$ is $J$-plurisubharmonic if 
and only if for every $J$-holomorphic disc $u: \Delta \rightarrow M$, 
$\rho \circ u$ is subharmonic (see \cite{iv-ro}).

\subsection{Pseudoconvex domains of finite type}

In this section, we recall some facts about pseudoconvex domains of finite type in four dimensional almost complex manifolds  (See \cite{ber} for more detailed facts).

Let $D=\{\rho<0\}$ is a smooth domain in $\R^4$. Assume that $\rho$ is $J$-plurisubharmonic on a neighborhood of $\overline{D}$
where the structure $J$ is defined on a fixed neighborhood $U$ of $\overline{D}$.
Moreover we suppose that the origin  is a boundary point of $D$.
\begin{defi}\label{defco} 
Let $u~: \left(\Delta,0\right)\rightarrow \left(\R^{4},0,J\right) $ be a $J$-holomorphic disc satisfying $u\left(0\right)=0$.
The order of contact $\delta_0\left(\partial D,u\right)$ with $\partial D$ at the origin is the degree of the first term 
in the Taylor expansion of $\rho \circ u$. We denote by $\delta\left(u\right)$ the multiplicity of $u$ at the origin. 
\end{defi}

We now define the (D'Angelo)  type and the regular type of the real hypersurface $\partial D$ at the origin.  

\begin{defi}\label{deftyp}\mbox{ }

 \begin{enumerate}
\item The (D'Angelo) type  of $\partial D$ at the origin is defined by:
$$\Delta^1\left(\partial D,0\right):=\sup \left\{\frac{\delta_0\left(\partial D,u\right)}{\delta\left(u\right)}, \mbox{ } 
u:\Delta\rightarrow \left(\R^{4},J\right) \mbox{  $J$-holomorphic nonconstant},  u\left(0\right)=0 \right\}.$$
The point $0$ is a point of finite (D'Angelo) type $2m$ if $\Delta^1\left(\partial D,0\right)=2m<+\infty$.
\item The   regular type of $\partial D$ at origin is defined by:
\begin{eqnarray*}
\Delta^1_{{\rm reg}}\left(\partial D,0\right)&:=&\sup \{\delta_0\left(\partial D,u\right), \mbox{ } 
u:\Delta\rightarrow \left(\R^{4},J\right) \mbox{  $J$-holomorphic },\\ 
& & \hspace{6,5cm} u\left(0\right)=0, d_0u\neq 0 \}.
\end{eqnarray*}
\end{enumerate}
\end{defi}
The  type condition  as defined in part 1 of Definition \ref{deftyp} was introduced by 
J.-P.D'Angelo \cite{d'a2} who proved that this coincides  with the regular type in complex manifolds
of dimension two.  It was proved in \cite{ber} that the (D'Angelo) type and  the regular type coincide in 
four dimensional almost complex manifolds.

\vspace{0,5cm}

In the next proposition, we describe locally the almost complex structure $J$ and the  defining function $\rho$ 
(see \cite{ber} for a proof).
\begin{prop}\label{prop1}
Let $D=\{\rho<0\}$ is a smooth domain in $\R^4$. Assume that $\rho$ is $J$-plurisubharmonic on a neighborhood of $\overline{D}$ where the structure $J$ is defined on a fixed neighborhood $U$ of $\overline{D}$. 
We suppose that the origin $0 \in partial D$ is a point of finite type $2m$. 
Then there is a local change of coordinates in a neighborhood of the origin such that, in the new coordinates:
\begin{equation}\label{eqstr}
J=\left(\begin{array}{ccccc} 
a_1& b_1 & 0 & 0 \\
c_1 & -a_1 & 0 & 0\\
0  & 0 & a_2 & b_2\\
0 & 0 & c_2 & -a_2\\ 
\end{array}\right)=J_{st}+O\left(|z_2|\right),
\end{equation} 

\begin{equation}\label{eqform}
\rho=\Re ez_2+H_{2m}\left(z_1,\overline{z_1}\right)+\widetilde{H}(z_1,z_2)+O\left(|z_1|^{2m+1}+|z_2||z_1|^m+|z_2|^2\right)
\end{equation} 
where  $H_{2m}$ is a homogeneous polynomial of degree $2m$, subharmonic which  is not harmonic and 
$$\widetilde{H}(z_1,z_2)=\Re e \displaystyle \sum_{k=1}^{m-1} \rho_{k} z_1^k z_2.$$
\end{prop}

A crucial tool for the study of pseudoholomorphic curves into pseudoconvex domains of finite type is the local peak $J$-plurisubharmonic functions which existence was proved in \cite{ber}.
\begin{theo}\label{theoex}
Let $D=\{\rho<0\}$ be a domain of finite type in an almost complex 
manifold $\left(M,J\right)$ of  dimension four. We suppose that 
$\rho$ is a smooth defining function of $D$, $J$-plurisubharmonic on a neighborhood of $\overline{D}$. 
Let $p \in \partial D$ be a boundary point. Then  there exist a function $\varphi_p$ and a neighborhood $U$ of $p$ such that $\varphi_p$ is 
continuous up to $\overline{D}\cap U$ and satisfies: 
\begin{enumerate}
\item $\varphi_p$ is $J$-plurisubharmonic on  $D\cap U$,
\item $\varphi_p\left(p\right)=0$,
\item $\varphi_p <0$ on  $\overline{D}\cap U\backslash\{p\}.$
\end{enumerate} 
Such a function is called a local peak $J$-plurisubharmonic function at $p$.
\end{theo}

\subsection{The Kobayashi pseudometric}
The existence of local pseudoholomorphic discs proved in \cite{ni-wo} 
allows to define the {\it Kobayashi pseudometric} $K_{\left(M,J\right)}$ for $p\in M$ and $v \in T_pM$ :
$$K_{\left(M,J\right)}\left(p,v\right):=\inf 
\left\{\frac{1}{r}>0, u : \Delta \rightarrow \left(M,J\right) 
\mbox{  $J$-holomorphic }, u\left(0\right)=p, d_{0}u\left(\partx\right)=rv\right\}.$$

Let $d_{\left(M,J\right)}$ be the integrated pseudodistance of $K_{\left(M,J\right)}$ defined by:
$$d_{\left(M,J\right)}\left(p,q\right): =\inf\left\{\int_0^1 
K_{\left(M,J\right)}\left(\gamma\left(t\right),\dot{\gamma}\left(t\right)\right)dt, \mbox{ }
\gamma : [0,1]\rightarrow M, \mbox{ }\gamma\left(0\right)=p, \gamma\left(1\right)=q\right\}.$$

We thus define :
\begin{defi}\label{defin}
\begin{enumerate}
\item The manifold $\left(M,J\right)$ is Kobayashi hyperbolic if the integrated pseudodistance $d_{\left(M,J\right)}$ is a distance.
\item Let $D\subset M$ be a domain in an almost complex manifold $\left(M,J\right)$. 
A point $p \in \partial D$ is said to be at finite distance from $q \in D$ if there is a sequence of points $q_j \in D$ converging to $p$ and whose Kobayashi distances $d_{\left(D,J\right)}(q_j,q)$ to $q$ stay bounded. 
Otherwise we say that the distance is infinite.
\end{enumerate}
\end{defi}

\vspace{0,7cm}

\section{Proof of Theorem \ref{theo1}}

In order to prove this theorem, we need the two following lemmas, where the dimension assumption is meaningful as it allows to find a coordinate system where the lines $\{z_1=c\}$ and $\{z_2=c'\}$ are almost complex submanifolds.

\begin{lem}\label{lemgrad}
Let $\Omega$ be an open subset 
of $(\R^{4},J)$ and let $J$ be an almost complex structure satisfying (\ref{eqstr}). 
Let $K$ be a compact subset of $\Omega$. There exists $\delta>0$
such that: for every $r\in [0,1)$
there exists a positive constant $C>0$
such that if $u=(u_1,u_2):\Delta \to \Omega$ is a $J$-holomorphic disc
with $u(\Delta)\subset K$, then
$$
|\nabla u_1(\zeta)| \leq C {\rm sup}_{t\in \Delta}|u_1(t)-u_1(0)|, \mbox{ and, } |\nabla u_2(\zeta)| \leq C {\rm sup}_{t\in \Delta}|u_2(t)-u_2(0)|,
$$
if ${\rm sup}_{t\in \Delta}\|u(t)-u(0)\|\leq \delta$ and
$\zeta \in \Delta_r$.
\end{lem}
This lemma is an anisotropic version of a result obtained by S.Ivashkovich, J.-P.Rosay in \cite{iv-ro}.
\begin{proof} 
Depending on $u(0) \in K$, one can make a linear change of variables such that in the new coordinates 
$J(u(0)) = J_{st}$. Set $\alpha_i:={\rm sup}_{t \in \Delta}~|u_i(t)-u_i(0)|$ for $i=1,2$ 
and consider the scaling map $\Lambda$ from $\R^4$ into itself defined
by  $\Lambda(z_1,z_2):=(\alpha_1^{-1} (z_1-u_1(0)),\alpha_2^{-1} (z_2-u_2(0)))$. 
Notice that the $\Lambda_*J$-holomorphic disc $\Lambda\circ u$ satisfies
$\Lambda\circ u (\Delta) \subseteq \Delta\times \Delta$.
If $\alpha_1$ and $\alpha_2$ are small enough  then $\Lambda_*J$ is close to $J_{st}$ as the structure $J$ has a diagonal form.
It follows from Proposition 2.3.6 in \cite{jcsi}, that for $|z|\leq r$ one gets for $i=1,2$, $|\nabla (\Lambda \circ u)_i(z)|\leq C$
for some positive constant $C$. Hence $|\nabla u_i(z)|\leq C\alpha_i$, as desired.
\end{proof}

A straightforward computation leads to this very useful lemma:
\begin{lem}\label{lemlev}
Assume that $J$ is an almost complex structure on $\R^4$ satisfying (\ref{eqstr}).
Then the Levi form of $\Re e z_2$ at  $z$ and $v=\left(X_1,Y_1,X_2,Y_2\right)$ is equal to  
\begin{eqnarray*}
\mathcal{L}_{J}\Re e z_2\left(z,v\right)&=& \left[\left(a_1-a_2\right)(z)\frac{\partial a_2}{\partial x_1}(z)
-c_2(z)\frac{\partial b_2}{\partial x_1}(z)+c_1(z)\frac{\partial a_2}{\partial y_1}(z)\right]X_1X_2+\\
&&\\
& &\left[\left(a_1+a_2\right)(z)\frac{\partial b_2}{\partial x_1}(z)
-b_2(z)\frac{\partial a_2}{\partial x_1}(z)+c_1(z)\frac{\partial b_2}{\partial y_1}(z)
\right]X_1Y_2+\\
&&\\
& &\left[-\left(a_1+a_2\right)(z)\frac{\partial a_2}{\partial y_1}(z)+
b_1(z)\frac{\partial a_2}{\partial x_1}(z)-c_2(z)\frac{\partial b_2}{\partial y_1}(z)
\right]Y_1X_2+\\
&&\\
& &\left[\left(a_2-a_1\right)(z)\frac{\partial b_2}{\partial y_1}(z)+
b_1(z)\frac{\partial b_2}{\partial x_1}(z)-b_2(z)\frac{\partial a_2}{\partial y_1}(z)
\right]Y_1Y_2+\\
&&\\
& & \left[\frac{\partial a_2}{\partial y_2}(z)-\frac{\partial b_2}{\partial x_2}(z)\right]\left(c_2(z)X_2^2-2a_2(z)X_2Y_2 -b_2(z)Y_2^2\right).
\end{eqnarray*}

\end{lem}

\vspace{1cm}

\begin{proof}[Proof of Theorem \ref{theo1}.]Let $U$ be a neighborhood of $0$ in $\R^4$. 
We assume that, on $U$, the structure $J$ satisfies (\ref{eqstr}) and  that the defining function has the local expression 
\begin{equation}\label{eqform2}
\rho=\Re e z_2+H_{2m}\left(z_1,\overline{z_1}\right)+O\left(|z_1|^{2m+1}+|z_2||z_1|^m+|z_2|^2\right)
\end{equation} 
where  $H_{2m}$ is a homogeneous polynomial of degree $2m$, subharmonic which  is not harmonic.

\vspace{0.8cm}

Consider for a positive number $\delta>0$, the following anisotropic polydisc:
$$Q(0,\delta):=\{z \in \C^2, |z_1|<\delta^{\frac{1}{2m}}, |z_2|<\delta \}.$$
Notice that since the defining function $\rho$ satisfies (\ref{eqform2}), then for a sufficiently small 
$\delta<1$, if $z \in Q(0,\delta)$ then we have ${\rm dist}\left(z,\partial D\right)\leq c\delta$ 
for some positive constant $c>0$.

\vspace{0.8cm}

Let $q'=(q_1',q_2') \in \partial D \cap U$ be a boundary point and let $\varphi_{q'}$ be a local peak $J$-plurisubharmonic function at
 the point $q'$. There is a positive constant $C_1$ such that  
\begin{equation}\label{eqine}
-C_1\|z-q'\|\leq \varphi_{q'}\left(z\right) \leq -C_2\Psi_{q'}\left(z\right),
\end{equation}
where 
$$\Psi_{q'}\left(z\right):=|z_1-q_1'|^{2m}+|z_2-q_2'|^2+|z_1-q_1'|^2|z_2-q_2'|^2$$ 
is a $J$-plurisubharmonic function on  $U$, shrinking $U$ if necessary.

\vspace{0.8cm}

Let  $u~: \Delta \rightarrow D\cap U$ be a $J$-holomorphic disc such that $u\left(0\right) \in Q(0,\delta)$ 
is sufficiently  close to the origin. In order to prove that the origin $0$ is at infinite distance 
from points in $D\cap U$, we want to provide the following estimates 
\begin{equation*}
|\nabla u_1\left(0\right)|\leq C\delta^{\frac{1}{2m}} \mbox{ and } \mbox{ } |\nabla u_2\left(0\right)|\leq C\delta
\end{equation*} for a positive constant 
$C>0$.

\vspace{0.5cm}

Let $q'=(q_1',q_2') \in \partial D$ be the unique boundary point such that $q'=u(0)+(0,\delta_u)$ for some positive $\delta_u$. Notice that $\delta_u$ is asymptotically equivalent to ${\rm dist}\left(u\left(0\right),\partial D\right)$.
According to the $J$-plurisubharmonicity of $\Psi_{q'}$,  we have  for $|\zeta|< r$ where $0<r<1$:
$$\Psi_{q'}\left(u\left(\zeta\right)\right)\leq \frac{A}{2\pi}\displaystyle 
\int_0^{2\pi}\Psi_{q'}\left(u\left(re^{i\theta}\right)\right)d\theta,$$
for an appropriate positive constant $A$.
Hence using (\ref{eqine}) and the $J$-plurisubharmonicity of the peak function $\varphi_{q'}$ we obtain:
$$\Psi_{q'}\left(u\left(\zeta\right)\right)\leq -\frac{A}{2\pi C_2}\displaystyle \int_0^{2\pi} \varphi_{q'}\left(u\left(re^{i\theta}\right)\right)d\theta
\leq -\frac{A}{C_2}\varphi_{q'}\left(u\left(0\right)\right).$$
Since  
$$|u_2\left(\zeta\right)-q_2'|^{2}\leq \Psi_{q'}\left(u\left(\zeta\right)\right)$$ 
and according to (\ref{eqine}) and to the fact that $\|u(0)-q'\|=\delta_u$ is asymptotically equivalent to ${\rm dist}\left(u\left(0\right),\partial D\right)$, 
we  obtain for a positive constant $C_3$: 
$$|u_2\left(\zeta\right)-q_2'|^{2}\leq C_3{\rm dist}\left(u\left(0\right),\partial D\right).$$
Hence, for some other positive constant $C_3$ we have: 
\begin{equation}\label{eqnorm1}
|u_2\left(\zeta\right)-u_2(0)|\leq C_3{\rm dist}\left(u\left(0\right),\partial D\right)^{\frac{1}{2}},
\end{equation}
for $|\zeta|< r$ where $0<r<1$.

For the same reason, we obtain the following estimate which hold for $|\zeta|< r<1$:
\begin{equation}\label{eqtan1}
|u_1\left(\zeta\right)-u_1(0)|\leq C_3{\rm dist}\left(u\left(0\right),\partial D\right)^{\frac{1}{2m}}.
\end{equation}

According to Lemma \ref{lemgrad}, inequalities (\ref{eqnorm1}) and (\ref{eqtan1}) imply:
\begin{equation}\label{eqnab}
|\nabla u_2(\zeta)|\leq B \sup_{|t|<r} |u_2\left(t\right)-u_2\left(0\right)|\leq  
C_3B{\rm dist}\left(u\left(0\right),\partial D\right)^{\frac{1}{2}} \leq C_4\delta^{\frac{1}{2}}.
\end{equation}
and,
\begin{equation}\label{eqfin1}
|\nabla u_1(0)|\leq B \sup_{|t|<r} |u_1\left(t\right)-u_1\left(0\right)|\leq  
C_3B{\rm dist}\left(u\left(0\right),\partial D\right)^{\frac{1}{2m}} \leq C_4\delta^{\frac{1}{2m}}.
\end{equation}
for  positive constants  $B$ and $C_4$. Notice that (\ref{eqfin1}) is the desired tangential estimate.  

\vspace{0.5cm}

In order to obtain the normal estimate $|\nabla u_2(0)|$, we will construct a negative harmonic function. 
To achieve this, we first need to control $\Re e u_2(\zeta)$ and $|\Delta {\Re e}~u_2(\zeta)|$ by $\delta$.

From  (\ref{eqnorm1}) and (\ref{eqtan1}) we obtain for $|\zeta|< r$:
\begin{equation}\label{eqdist3}
|u_1(\zeta)|\leq |u_1\left(\zeta\right)-u_1(0)|+|u_1(0)| \leq  C_3{\rm dist}\left(u\left(0\right),\partial D\right)^{\frac{1}{2m}}+\delta^{\frac{1}{2m}} \leq C_5\delta^{\frac{1}{2m}}.
\end{equation}
and,
\begin{equation}\label{eqdist4}
|u_2(\zeta)|\leq |u_2\left(\zeta\right)-u_2(0)|+|u_2(0)| \leq C_3{\rm dist}\left(u\left(0\right),\partial D\right)^{\frac{1}{2}} + \delta \leq C_{5}\delta^{\frac{1}{2}}.
\end{equation}
for a positive constant $C_{5}$. According to these estimates, and since the defining function satisfies
(\ref{eqform2}), we obtain 
\begin{equation}\label{eqdist5}
\Re e  u_2(\zeta)\leq \left|H_{2m}\left(u_1(\zeta),\overline{u_1(\zeta)}\right)\right|+
O\left(|u_1(\zeta)|^{2m+1}+|u_2(\zeta)||u_1(\zeta)|^m+|u_2(\zeta)|^2\right)\leq C_{6}\delta. 
\end{equation}
if $|\zeta|<r$, where $C_{6}>0$.
Due to Proposition \ref{proplevi} applied to the function $\Re e z_2$, we have: 
\begin{equation}
\Delta {\Re e}~u_2(\zeta)=\mathcal{L}_{J}\Re e z_2\left(u(\zeta),\frac{\partial u}{\partial x}(\zeta)\right).
\end{equation}
According to Lemma \ref{lemlev}, (\ref{eqnab}), (\ref{eqdist4}) and to the fact that $J(z)=J_{st}+O(|z_2|)$, this implies
\begin{equation}\label{eqlap2}
|\Delta {\Re e}~u_2(\zeta)|\le C_{7}\delta
\end{equation}if $|\zeta|<r$, and a positive constant $C_{7}$.

We still denote by $\Delta {\Re e}~u_2$, the function equal to $\Delta {\Re e}~u_2$ on $\Delta_{r}$, 
and $0$ elsewhere. And consider the following harmonic function
\begin{equation}\label{eqg2}
g(\zeta):=\Re e u_2(\zeta)-\left[\frac{1}{2\pi}\ln |\zeta|*\Delta {\Re e}~u_2\right](\zeta)
-(C_{6}+C_{7})\delta.
\end{equation}
From (\ref{eqlap2}), we have 
$$\left|\frac{1}{2\pi}\ln |\zeta|*\Delta {\Re e}~u_2~\right|\leq C_{7}\delta.$$
Thus $g$ is a negative function and $|g(0)|\leq 2(C_{6}+C_{7})\delta$.
The classical Schwarz Lemma for negative harmonic functions gives
$|\nabla g(0)|\leq 2 |g(0)|$. Hence we obtain the following estimate:  
$$|\nabla {\Re e}~u_2(0)|\leq |\nabla g(0)|+{\rm Sup}~|\Delta {\Re e}~u_2|
\leq 2|g(0)|+C_7\delta\leq (4C_{6}+5C_{7})\delta.$$

Moreover, the $J$-holomorphy equation for the disc $u$ implies:
\begin{eqnarray*}
a_2(u(0))\frac{\partial \Re e u_2}{\partial x}(0)+b_2(u(0))\frac{\partial \Im m u_2}{\partial x}(0)&=&\frac{\partial \Re e u_2}{\partial y}(0)\\
c_2(u(0))\frac{\partial \Re e u_2}{\partial x}(0)-a_2(u(0))\frac{\partial \Im m u_2}{\partial x}(0)&=&\frac{\partial \Im m u_2}{\partial y}(0),\\
\end{eqnarray*}
so that 
$$|\nabla {\Im m}~u_2(0)|\leq C_{8} |\nabla {\Re e}~u_2(0)|,$$
for some positive constant $C_{8}$. 

We finally get  the  normal estimate:
\begin{equation}\label{eqfin2}
|\nabla u_2(0)|\leq C\delta.
\end{equation}

\vspace{0.8cm}

Let us show how  (\ref{eqfin1}) and (\ref{eqfin2}) imply that the origin $0$ is at infinite distance from points in $D\cap U$. Consider the function: 
$$\chi(z):=|z_1|^{4m}+|z_2|^2.$$
Since  $u(0) \in Q(0, \chi(u(0))^{\frac{1}{2}})$, it follows from the estimates 
(\ref{eqfin1}) and  (\ref{eqfin2}) that  
$|\nabla u_1(0)|\leq C\chi(u(0))^{\frac{1}{4m}}$ and $|\nabla u_2(0)|\leq C\chi(u(0))^{\frac{1}{2}}$ for some positive constant $C$ and thus:  
\begin{eqnarray*}
|\nabla (\chi\circ u )(0)| &\leq & 4m|\nabla u_1(0)||u_1(0)|^{4m-1}+2|\nabla u_2(0)||u_2(0)|\\
&\leq & 4mC\chi(u(0))^{\frac{1}{4m}}(\chi(u(0))^{\frac{1}{4m}})^{4m-1}+2C\chi(u(0))^{\frac{1}{2}}\chi(u(0))^{\frac{1}{2}}\\
&\leq & (4mC+2C)(\chi\circ u )(0)\\
\end{eqnarray*}
which achieves the proof of Theorem \ref{theo1} by an integration argument (see Lemma 1.1 in \cite{iv-ro}).

\end{proof}

\end{document}